\documentclass[english,11pt,reqno]{smfart}

\usepackage{color}
\usepackage{amsthm}
\usepackage{amsmath}
\usepackage{amsfonts}
\usepackage{amstext}
\usepackage[latin1]{inputenc}
\usepackage{amscd}
\usepackage{latexsym}
\usepackage{bm}
\usepackage{amssymb}
\usepackage[all]{xy}
\usepackage{euscript}
\usepackage{a4wide}
\usepackage{amsmath,amssymb,graphicx}
\usepackage{amssymb}
\usepackage{amsmath}
\usepackage{mathrsfs,mathtools}

\newtheorem{theorem}{Theorem}

\newtheorem{property}[theorem]{Property}
\newtheorem{lemma}[theorem]{Lemma}
\newtheorem{corollary}[theorem]{Corollary}
\newtheorem{example}[theorem]{Example}
\newtheorem{definition}[theorem]{Definition}

\def\di{\displaystyle}

\def\eps{\varepsilon}
\def\ie{\textit{i.e.} }
\def\({\left(} 
\def\){\right)}

\newcommand{\N}{\mathbb{N}}
\newcommand{\R}{\mathbb{R}}

\newcommand{\LL}{\mathcal{L}}
\newcommand{\CC}{\mathscr{C}}

\newcommand{\cDM}{{}_{{\rm c}} D^\alpha_-}
\newcommand{\cDP}{{}_{{\rm c}} D^\alpha_+}
\newcommand{\DM}{D^\alpha_-}
\newcommand{\DP}{D^\alpha_+}

\newcommand{\fonction}[5]{\begin{array}[t]{lrcl}#1 :&#2 &\longrightarrow &#3\\&#4& \longmapsto &#5 \end{array}}

\newcommand{\fonctionsansdef}[3]{\begin{array}[t]{lrcl}#1 :&#2 &\longrightarrow &#3 \end{array}}

\newcommand{\abs}[1]{\left\vert #1 \right\vert} 

\DeclareMathOperator*{\argmin}{arg\,min}

\begin{document}
\title[]{\textbf{A class of fractional optimal control problems and fractional Pontryagin's systems. Existence of a fractional Noether's theorem}}
\author{Lo\"ic Bourdin}
\address{Laboratoire de Math\'ematiques et de leurs Applications - Pau (LMAP). UMR CNRS 5142. Universit\'e de Pau et des Pays de l'Adour.}
\email{bourdin.l@etud.univ-pau.fr}
\maketitle

\begin{abstract}
In this paper, we study a class of fractional optimal control problems. A necessary condition for the existence of an optimal control is provided in \cite{agra2,agra3,torr,torr2,jeli} and references therein. It is commonly given as the existence of a solution of a fractional Pontryagin's system and the proof is based on the introduction of a Lagrange multiplier. Assuming an additional condition on these problems, we suggest a new presentation of this result with a proof using only classical mathematical tools adapted to the fractional case: \textit{calculus of variations}, \textit{Gronwall's Lemma}, \textit{Cauchy-Lipschitz Theorem} and \textit{stability under perturbations of differential equations}.

In this paper, we furthermore provide a way in order to transit from a classical optimal control problem to its fractional version via the Stanislavsky's formalism. We also solve a strict fractional example allowing to test numerical schemes. Finally, we state a fractional Noether's theorem giving the existence of an explicit constant of motion for fractional Pontryagin's systems admitting a symmetry.
\end{abstract}



\section*{Introduction}
The control theory is the analysis of controlled dynamical systems. These systems are varied: they can be differential, stochastic or discrete. The optimal control theory concerns the determination of a control optimizing a certain cost. Consequently, this theory is strongly connected to the 18${}^{\text{th}}$ century classical mechanic (variational principles, Euler-Lagrange equations, etc., see \cite{arno,gamk,mart}). Since the second world war, this theory has a considerable development and one can find applications in many domains: celestial mechanic \cite{trel2}, biology \cite{burl}, hydroelectricity \cite{cant}, economy \cite{dema,dmit,fedo}, etc. The subject is widely treated and one can find a lot of references on the subject, see for example \cite{brys,evan,hest,trel}. 

The fractional calculus, \textit{i.e.} the mathematical field dealing with the generalization of the derivative to any real order, plays an increasing role in many varied domains as economy \cite{comt} or probability \cite{levy,stan}. Fractional derivatives also appear in many fields of Physics (see \cite{hilf3}):  wave mechanic \cite{alme}, viscoelasticity \cite{bagl}, thermodynamics \cite{hilf2}, fluid mechanic in heterogeneous media \cite{hilf,neel,neel2}, etc. A natural question then arises: \textit{can we develop optimal control theories for fractional differential systems?}

Recently, a subtopic of the fractional calculus gains importance: it concerns the variational principles on functionals involving fractional derivatives. This leads to the statement of fractional Euler-Lagrange equations, see \cite{agra,bale2,riew}. A direct consequence is the emergence of works concerning a particular class of fractional optimal control problems, see \cite{agra2,agra3,torr,torr2,jeli} and references therein. These studies usually use the Lagrange multiplier technique allowing to write these problems as problems of optimization without constraint of \textit{augmented functionals}. With a calculus of variations, authors then obtain a necessary condition for the existence of an optimal control. This condition is commonly given as the existence of a solution of a system of fractional differential equations called \textit{fractional Pontryagin's system}. \\

In this paper, we first give a new presentation of this result. Precisely, making an additional assumption (see Condition \eqref{condf}), we rewrite \textit{directly} these fractional optimal control problems as simpler problems of optimization without constraint of functionals \textit{depending only on the control}. With a calculus of variations, we finally conclude to the same necessary condition but with a new presentation. Although the method here used is considerably inspired by the Lagrange multiplier technique, it allows us to give a complete proof using only classical mathematical tools adapted to the fractional case: \textit{calculus of variations}, \textit{Gronwall's Lemma}, \textit{Cauchy-Lipschitz Theorem} and \textit{stability under perturbations of differential equations}. 

Nevertheless, the explicit computation of controls satisfying the above necessary condition needs the resolution of the fractional Pontryagin's system which is a main drawback. Indeed, solving a fractional differential equation is in general very difficult. Consequently, in this paper, we suggest a deviously way in order to get informations on the solutions of a fractional Pontryagin's system. Precisely, we study the existence of classical conservation laws, \textit{i.e.} functions which are constant on each solution. Indeed, constants of motion, generally associated to physical quantities, give strong informations on the solutions in the phase space for example. Moreover, they also can be used in order to reduce or integrate the system by quadrature. Previous results in this direction have been obtained by Torres and Frederico in \cite{torr,torr2}. However, in each of these papers, the conservation law is not explicit but implicitly defined by a functional relation. In this paper, inspired by a recent result obtained in \cite{bour2}, we prove a fractional Noether's theorem providing an explicit conservation law for fractional Pontryagin's systems exhibiting a symmetry. 

In this paper, we also suggest:
\begin{itemize}
\item a link between classical optimal control problems and their fractional versions via the Stanislavsky's formalism (see \cite{iniz,stan});
\item a solved fractional example allowing to test numerical schemes in the strict fractional case.
\end{itemize}
~\\
The paper is organized as follows. Firstly, we remind some classical definitions and results concerning fractional calculus (Section \ref{section1}). Then, we present the class of fractional optimal control problems studied and the usual strategy using the Lagrange multiplier (Section \ref{section21}). Then, assuming an additional condition on the problem, we develop a new proof leading to the same result with a different presentation, see Section \ref{section22}. We then suggest a way in order to transit from a classical optimal control problem to its fractional version via the Stanislavsky's formalism (Section \ref{section3}). Subsequently, we suggest some examples with, in particular, a solved fractional example allowing to test numerical schemes in the strict fractional case (Section \ref{section4}). Finally, in Section \ref{section5}, we state a fractional Noether's theorem for fractional Pontryagin's systems admitting a symmetry. Technical proofs of Lemmas are provided in Appendix \ref{appA}.

\parindent 0pt

\section{Usual definitions and results about fractional calculus}\label{section1}
Let us introduce the following notations available in the whole paper. Let $a < b$ be two reals, let $d$, $m \in \N^*$ denote two dimensions and let $\Vert \cdot \Vert$ be the euclidean norm of $\R^d$ and $\R^m$.

\subsection{Fractional operators of Riemann-Liouville and Caputo}\label{section11}
The fractional calculus concerns the extension of the usual notion of derivative from non-negative integer orders to any real order. Since 1695, numerous notions of fractional derivatives emerge over the year, see \cite{kilb,podl,samk}. In this paper, we only use the notions of fractional integrals and derivatives in the sense of Riemann-Liouville (1847) and Caputo (1967) whose definitions are recalled in this section. We refer to \cite{kilb,podl,samk} for more details. \\

Let $g \in \CC^0 ([a,b],\R^d)$ and $\alpha > 0$. The left (resp. right) fractional integral in the sense of Riemann-Liouville with inferior limit $a$ (resp. superior limit $b$) of order $ \alpha $ of $g$ is defined by:
\begin{equation}
\forall t \in ]a,b], \; I^{\alpha}_- g (t) := \dfrac{1}{\Gamma (\alpha)} \di \int_a^t (t-y)^{\alpha -1} g(y) \; dy
\end{equation}
respectively:
\begin{equation}
\forall t \in [a,b[, \; I^{\alpha}_+ g (t) := \dfrac{1}{\Gamma (\alpha)} \di \int_t^b (y-t)^{\alpha -1} g(y) \; dy,
\end{equation}
where $\Gamma$ denotes the Euler's Gamma function. Let us remind that $I^\alpha_- g$ (resp. $I^\alpha_+ g$) is continuous and can be continuously extended by $0$ in $t=a$ (resp. $t=b$). Let us note that $I^1_- g$ (resp. $-I^1_+ g$) coincides with the anti-derivative of $g$ vanishing in $t=a$ (resp. $t=b$). For $\alpha =0$, let $I^0_- g = I^0_+ g = g$. \\

Now, let us consider $0 < \alpha \leq 1$. The left (resp. right) fractional derivative in the sense of Riemann-Liouville with inferior limit $a$ (resp. superior limit $b$) of order $\alpha $ of $g$ is then given by:
\begin{equation}
\forall t \in ]a,b], \; \DM g(t) := \dfrac{d}{dt} \big( I^{1-\alpha}_- g \big) (t) \quad \Big( \text{resp.} \quad \forall t \in [a,b[, \; \DP g(t) := -\dfrac{d}{dt} \big( I^{1-\alpha}_+ g \big) (t) \Big),
\end{equation}
provided that the right side terms are defined. \\

In the Riemann-Liouville sense, the strict fractional derivative of a constant is not zero. Caputo then suggests the following definition. For $0 < \alpha \leq 1$, the left (resp. right) fractional derivative in the sense of Caputo with inferior limit $a$ (resp. superior limit $b$) of order $\alpha $ of $g$ is given by:
\begin{equation}
\forall t \in ]a,b], \; \cDM g(t) := \DM \big( g - g(a) \big) (t) \quad \Big( \text{resp.} \quad \forall t \in [a,b[, \; \cDP g(t) := \DP \big( g - g(b) \big) (t) \Big),
\end{equation}
provided that the right side terms are defined. Let us note that if $g(a)=0$ (resp. $g(b)=0$), then $\cDM g = \DM g$ (resp. $\cDP g =  \DP g$).\\

In the \textit{classical case} $\alpha = 1$, the fractional derivatives of Riemann-Liouville and Caputo both coincide with the classical derivative. Precisely, modulo a $(-1)$ term in the right case, we have $D^1_- = {}_{\text{c}}D^1_- = - D^1_+ = - {}_{\text{c}}D^1_+ = d/dt$. \\

Finally, let us remind the following important result. If $g \in \CC^1 ([a,b],\R^d)$ and $0 < \alpha \leq 1$, then we have:
\begin{equation}\label{eq11-1}
\forall t \in ]a,b], \; \cDM g(t) =  I^{1-\alpha}_- \dot{g} (t) \quad \text{and} \quad \forall t \in [a,b[, \; \cDP g(t) = - I^{1-\alpha}_+ \dot{g} (t) ,
\end{equation}
where $\dot{g} $ denotes the derivative of $g$. Consequently, in this case, we have $\cDM g$ (resp. $\cDP g$) is continuous and can be continuously extended in $t=a$ (resp. $t=b$). 

\subsection{Some properties and results about the fractional operators}\label{section12}
Let us precise that the properties and results developed in this section are well-known and already proved in many references, see \cite{kilb,samk} for example. However, these results are not always exactly presented as we need. In this case, we give a detailed proof for the reader's convenience. \\

First, we remind two basic results concerning the fractional integrals. They are proved in \cite{kilb,samk} both using the Fubini's theorem. The first one yields the \textit{semi-group property} of the fractional integral operators:

\begin{property}\label{prop1}
Let $g \in \CC^0 ([a,b],\R^d)$ and $\alpha_1$, $\alpha_2 \geq 0$. Then, we have:
\begin{equation}
\forall t \in [a,b], \; I^{\alpha_1}_- \circ I^{\alpha_2}_- g (t) = I^{\alpha_1 + \alpha_2}_- g(t) \quad \text{and} \quad I^{\alpha_1}_+ \circ I^{\alpha_2}_+ g (t) = I^{\alpha_1 + \alpha_2}_+ g(t).
\end{equation}
\end{property}

The following second property is occasionally called \textit{fractional integration by parts}. It is very useful for calculus of variations involving fractional derivatives:

\begin{property}\label{prop2}
For any $g_1$, $g_2 \in \CC^0 ([a,b],\R^d)$ and any $\alpha \geq 0$, we have:
\begin{equation}
\di \int_a^b I^\alpha_- g_1 \cdot g_2 \; dt = \di \int_a^b  g_1 \cdot I^\alpha_+ g_2 \; dt.
\end{equation}
\end{property}

Let us introduce the following convention: a function defined on $]a,b]$ (resp. $[a,b[$) is said to be an element of $\CC^0 ([a,b],\R^d)$ if and only if it is continuous on $]a,b]$ (resp. $[a,b[$) and if it can be continuously extended in $t=a$ (resp. $t=b$). From Section \ref{section11}, we can give the following examples:
\begin{itemize}
\item for $g \in \CC^0 ([a,b],\R^d)$ and $\alpha \geq 0$, we have $I^\alpha_- g$, $I^\alpha_+ g \in \CC^0 ([a,b],\R^d)$;
\item for $g \in \CC^1 ([a,b],\R^d)$ and $0 < \alpha \leq 1$, we have $\cDM g$, $\cDP g \in \CC^0 ([a,b],\R^d)$.
\end{itemize}
Now, we prove some results of composition between the left fractional operators. One can easily derive the analogous versions for the right ones. 

\begin{property}\label{prop3}
Let $g \in \CC^0 ([a,b],\R^d)$ and $0 < \alpha \leq 1$. Then, $\cDM \circ I^\alpha_- g $ is an element of $\CC^0 ([a,b],\R^d)$ and for any $t \in [a,b] $, we have:
\begin{equation}
\cDM \circ I^\alpha_- g(t) = g(t).
\end{equation}
Let us assume additionally that $\cDM g \in \CC^0 ([a,b],\R^d)$. Then, $ I^\alpha_- \circ \cDM g $ is an element of $\CC^0 ([a,b],\R^d)$ and for any $t \in [a,b]$, we have:
\begin{equation}
I^\alpha_- \circ \cDM g (t) = g(t)-g(a).
\end{equation}
\end{property}

\begin{proof}
Let us prove the first result. Since $I^\alpha_- g (a) =0$, we have for any $t \in ]a,b]$:
\begin{equation}
\cDM \circ I^\alpha_- g(t) = \DM \circ I^\alpha_- g(t) = \dfrac{d}{dt} \circ I^{1-\alpha}_- \circ I^\alpha_- g (t) = \dfrac{d}{dt} \big( I^1_- g \big) (t) = g(t).
\end{equation}
Hence, $\cDM \circ I^\alpha_- g $ is continuous on $]a,b]$ and can be continuously extended by $g(a)$ in $t=a$. Now, let us prove the second result. It is obvious for $\alpha =1$. Now, let us consider $0 < \alpha < 1$. Since $\cDM g \in \CC^0 ([a,b],\R^d)$, we have $I^{1-\alpha}_- \big( g - g(a) \big) \in \CC^1 ([a,b],\R^d)$. Combining the first result and Equality \eqref{eq11-1}, we have for any $t \in ]a,b]$:
\begin{multline}
g(t)-g(a) = {}_{\text{c}} D^{1-\alpha}_- \circ I^{1-\alpha}_- \big( g - g(a) \big) (t) = I^\alpha_- \circ \dfrac{d}{dt} \circ I^{1-\alpha}_- \big( g - g(a) \big) (t) \\ = I^\alpha_- \circ \DM \big( g - g(a) \big) (t) = I^\alpha_- \circ \cDM g (t).
\end{multline}
Hence $ I^\alpha_- \circ \cDM g $ is continuous on $]a,b]$ and can be continuously extended by $0$ in $t=a$. The proof is completed.
\end{proof}

Finally, we prove the following fractional Cauchy-Lipschitz type theorem. Let us precise that a version of this theorem is proved in \cite[Part 3.5.1, Corollary 3.26, p.205]{kilb}. It gives the existence and the uniqueness of a global solution of a fractional differential equation.
\begin{theorem}[Fractional Cauchy-Lipschitz theorem]\label{thmfcl}
Let $0 < \alpha \leq 1$, $F \in \CC^0(\R^d \times [a,b], \R^d)$ and $A \in \R^d$. Let us assume that $F$ satisfies the following Lipschitz type condition:
\begin{equation}
\exists K \in \R, \; \forall (x_1,x_2,t) \in (\R^d)^2 \times [a,b], \; \Vert F(x_1,t)-F(x_2,t) \Vert \leq K \Vert x_1 - x_2 \Vert .
\end{equation}
Then, the following fractional Cauchy problem:
\begin{equation}\label{eqcp}
 \left\lbrace \begin{array}{l}
 		\cDM g = F(g,t) \\
 		g(a) = A,
        \end{array}
\right.
\end{equation}
has an unique solution in $\CC^0 ([a,b],\R^d)$. The solution $g$ is an element of $ \CC^{[\alpha]} ([a,b],\R^d)$ (where $[\alpha]$ denotes the floor of $\alpha$) and $g$ satisfies:
\begin{equation}\label{eqcp2}
\forall t \in [a,b], \; g(t) = A + I^\alpha_- \big( F(g,t) \big) (t).
\end{equation}
\end{theorem}

\begin{proof}
Firstly, let $g \in \CC^0 ([a,b],\R^d)$. From Property \ref{prop3}, we conclude that $g$ is solution of the fractional Cauchy problem \eqref{eqcp} if and only if $g$ satisfies Equality \eqref{eqcp2}. Now, let us define:
\begin{equation}
\fonction{\varphi}{\CC^0 ([a,b],\R^d)}{\CC^0 ([a,b],\R^d)}{g}{\fonction{\varphi (g)}{[a,b]}{\R^d}{t}{A + I^\alpha_- \big( F(g,t) \big) (t).}}
\end{equation}
By induction, we prove that for any $n \in \N^*$ and any $g_1$, $g_2 \in \CC^0 ([a,b],\R^d)$, we have:
\begin{equation}
\forall t \in [a,b], \; \Vert \varphi^n (g_1)(t) - \varphi^n (g_2)(t) \Vert \leq \dfrac{ K^n(t-a)^{n\alpha}}{\Gamma (1+n\alpha) } \Vert g_1 - g_2 \Vert_{\infty}.
\end{equation}
Consequently, we have for any $n \in \N^*$ and any $g_1$, $g_2 \in \CC^0 ([a,b],\R^d)$:
\begin{equation}
\Vert \varphi^n (g_1) - \varphi^n (g_2) \Vert_{\infty} \leq \dfrac{ K^n(b-a)^{n\alpha}}{\Gamma (1+n\alpha) } \Vert g_1 - g_2 \Vert_{\infty}.
\end{equation}
Since $ K^n(b-a)^{n\alpha} / \Gamma (1+n\alpha) \longrightarrow 0$, there exists $n \in \N^*$ such that $\varphi^n$ is a contraction. Consequently, $\varphi$ admits an unique fix point in $\CC^0 ([a,b],\R^d)$. The proof is completed.
\end{proof}

\section{A class of fractional optimal control problems}\label{section2}
From now and for all the rest of the paper, we consider $0 < \alpha \leq 1$ and $A \in \R^d$.

\subsection{Presentation of the problem and usual strategy}\label{section21}
In this section, we first give a brief presentation of the class of fractional optimal control problems interesting us in this paper. Finally, we briefly remind the usual strategy (using the Lagrange multiplier technique) leading to a necessary condition for the existence of an optimal control, see \cite{agra2,agra3,torr,torr2,jeli} and references therein. \\

An optimal control problem concerns the optimization of a quantity depending on parameters given by a controlled system. Precisely, the aim is to find a control optimizing a certain cost. Such a control is called \textit{optimal control}. Many of these studies lead to the use of the Lagrange multiplier technique. For example, we refer to \cite{berg,favi,gunz2,gunz} for controlled systems coded by ordinary differential equations. \\

In this paper, we work in the following framework. We are interested in systems controlled by the following fractional Cauchy problem:
\begin{equation}\label{eq21-1}
\left\lbrace \begin{array}{l}
\cDM q = f (q,u,t) \\
q(a) = A,
\end{array} \right.
\end{equation}
where $u$ denotes the control. Our aim is to find a control $u$ optimizing a quantity of the form:
\begin{equation}\label{eq21-2}
\di \int_a^b L(q,u,t) \; dt ,
\end{equation}
where $q$ is solution of \eqref{eq21-1}. \\

The common strategy is first to write this problem as a problem of optimization under constraint: 
\begin{equation}\label{eq21-3}
\argmin\limits_{\substack{(q,u) \; \text{satisfying}\; \eqref{eq21-1}}} \; \di \int_a^b L(q,u,t) \; dt.
\end{equation}
Then, the Lagrange multiplier technique consists in the study of the critical points of the following \textit{augmented functional}:
\begin{equation}
(q,u,p) \longmapsto \di \int_a^b L(q,u,t)-p \cdot \big( \cDM q - f(q,u,t) \big) \; dt ,
\end{equation}
where $p$ is commonly called \textit{Lagrange multiplier}. Finally, with a calculus of variations, such a strategy leads to the following result: a necessary condition for $(q,u)$ to be a solution of \eqref{eq21-3} is that there exists a function $p$ such that the following \textit{fractional Pontryagin's system} holds:
\begin{equation}
 \left\lbrace \begin{array}{l}
 		\cDM q = \dfrac{\partial H}{\partial w} (q,u,p,t) \\[10pt]
 		\DP p = \dfrac{\partial H}{\partial x} (q,u,p,t) \\[10pt]
	    \dfrac{\partial H}{\partial v} (q,u,p,t) = 0 \\[10pt]
	    \big( q(a),p(b) \big) = ( A,0 ),
	    \end{array}
\right.
\end{equation}
where $ H(x,v,w,t) = L(x,v,t) + w  \cdot f(x,v,t) $. We refer to \cite{agra2,agra3,torr,torr2,jeli} and references therein for more details. \\

In this paper, we are going to assume that $f$ satisfies a Lipschitz type condition, see Condition \eqref{condf}. This assumption will allow us to write \textit{directly} the initial problem as a simpler problem of optimization without constraint of the initial functional which is then \textit{only dependent of the control} $u$. Finally, a simple calculus of variations leads us to the same result but with a new presentation. Let us remind that the method here developed is widely inspired from the Lagrange multiplier technique. Nevertheless, it allows us to give a complete proof only using classical mathematical tools adapted to the fractional case: \textit{calculus of variations}, \textit{Gronwall's Lemma}, \textit{Cauchy-Lipschitz Theorem} and \textit{stability under perturbations of differential equations}.

\subsection{New presentation of the result}\label{section22}
In this section, we first give rigorously the definitions concerning the class of fractional optimal control problems described in Section \ref{section21}:
\begin{itemize}
\item The elements denoted $u \in \CC^0 ([a,b],\R^m)$ are called \textit{controls};
\item Let $f$ be a $\CC^2$ function of the form:
\begin{equation}
\fonction{f}{\R^d \times \R^m \times [a,b]}{\R^d}{(x,v,t)}{f(x,v,t).}
\end{equation}
It is commonly called the \textit{constraint function}. We assume that $f$ satisfies the following Lipschitz type condition. There exists $M \geq 0$ such that:
\begin{equation}\tag{$f_x$ lip}\label{condf}
\forall (x_1,x_2,v,t) \in (\R^d)^2 \times \R^m \times [a,b], \; \Vert f(x_1,v,t) - f(x_2,v,t) \Vert \leq M \Vert x_1 - x_2 \Vert ;
\end{equation}
\item For any control $u$, let $q^{u,\alpha} \in \CC^{[\alpha]} ([a,b],\R^d)$ denote the unique global solution of the following fractional Cauchy problem:
\begin{equation}\tag{CP${}^\alpha_q$}\label{eqcpq}
 \left\lbrace \begin{array}{l}
         \cDM q =f(q,u,t)\\
	     q(a) = A.
        \end{array}
\right.
\end{equation}
$q^{u,\alpha}$ is commonly called the \textit{state variable} associated to $u$. Its existence and its uniqueness are provided by Theorem \ref{thmfcl} and Condition \eqref{condf};
\item Finally, with this condition on $f$, the fractional optimal control problem described in Section \ref{section21} can be rewritten as the simpler problem of optimization of the following \textit{cost functional}:
\begin{equation}
\fonction{\LL^\alpha}{\CC^0 ([a,b],\R^m )}{\R}{u}{\di \int_{a}^{b} L ( q^{u,\alpha},u,t ) \; dt ,} 
\end{equation}
where $L$ is a \textit{Lagrangian}, \textit{i.e.} a $\CC^2$ application of the form:
\begin{equation}
\fonction{L}{\R^d \times \R^m \times [a,b]}{\R}{(x,v,t)}{L(x,v,t).}
\end{equation}
\end{itemize}
Hence, the existence and the uniqueness of $q^{u,\alpha}$ for any control $u$ allow us to rewrite \textit{directly} the initial problem as a simpler problem of optimization without constraint of the cost functional $\LL^\alpha$: we do not need to introduce an augmented functional with a Lagrange multiplier. Moreover, let us note that $\LL^\alpha$ is only dependent of the control. \\

A control optimizing $\LL^\alpha$ is called \textit{optimal control}. A necessary condition for a control $u$ to be optimal is to be a \textit{critical point} of $\LL^\alpha$, \textit{i.e.} to satisfy:
\begin{equation}
\forall \bar{u} \in  \CC^0 ([a,b],\R^m ), \;  D\LL^\alpha(u)(\bar{u}) := \lim\limits_{\eps \rightarrow 0} \dfrac{\LL^\alpha(u+\eps \bar{u})-\LL^\alpha(u)}{\eps} = 0. 
\end{equation}
In the following, we then focus on the characterization of the critical points of $\LL^\alpha$. With an usual calculus of variations, we obtain the following Lemma \ref{lem1} giving explicitly the value of the G\^ateaux derivative of $\LL^\alpha$:
\begin{lemma}\label{lem1}
Let $u$, $\bar{u} \in \CC^0 ([a,b],\R^m )$. Then, the following equality holds:
\begin{equation}
D\LL^\alpha(u)(\bar{u}) = \di \int_a^b \dfrac{\partial L}{\partial x} (q^{u,\alpha},u,t) \cdot \bar{q} + \dfrac{\partial L}{\partial v} (q^{u,\alpha},u,t) \cdot \bar{u} \; dt,
\end{equation}
where $\bar{q} \in \CC^{[\alpha]} ( [a,b], \R^d )$ is the unique global solution of the following linearised Cauchy problem:
\begin{equation}\label{eqlcpq}\tag{LCP${}^\alpha_{\bar{q}}$}
 \left\lbrace \begin{array}{l}
 		\cDM \bar{q} = \dfrac{\partial f}{\partial x} (q^{u,\alpha},u,t) \times \bar{q} + \dfrac{\partial f}{\partial v} (q^{u,\alpha},u,t) \times \bar{u} \\[10pt]
 		\bar{q}(a) = 0.
        \end{array}
\right. 
\end{equation}
\end{lemma}

\begin{proof}
See Appendix \ref{appAd}.
\end{proof}

This last result does not lead to a characterization of the critical points of $\LL^\alpha$ yet. Then, let us introduce the following elements stemming from the Lagrange multiplier technique:
\begin{itemize}
\item Let $H$ be the following application
\begin{equation}
\fonction{H}{\R^d \times \R^m \times \R^d \times [a,b]}{\R}{(x,v,w,t)}{L(x,v,t)+w \cdot f(x,v,t).}
\end{equation}
$H$ is commonly called the \textit{Hamiltonian} associated to the Lagrangian $L$ and the constraint function $f$;
\item For any control $u$, let $p^{u,\alpha} \in \CC^{[\alpha]} ( [a,b], \R^d )$ denote the unique global solution of the following fractional Cauchy problem:
\begin{equation}\label{eqcpp}\tag{CP${}^\alpha_p$}
 \left\lbrace \begin{array}{l}
         \cDP p = \dfrac{\partial H}{\partial x}(q^{u,\alpha},u,p,t) = \dfrac{\partial L}{\partial x}(q^{u,\alpha},u,t) + \left( \dfrac{\partial f}{\partial x}(q^{u,\alpha},u,t)\right)^T \times p \\[10pt]
	     p(b) = 0.
        \end{array}
\right.
\end{equation}
$p^{u,\alpha}$ is commonly called the \textit{adjoint variable} associated to $u$. Its existence and its uniqueness are provided by the analogous version of Theorem \ref{thmfcl} for right fractional derivative. Since $p^{u,\alpha} (b) = 0$, we can write $\cDP p^{u,\alpha} = \DP p^{u,\alpha}$.
\end{itemize}
Let us note that, for any control $u$, the couple $(q^{u,\alpha},p^{u,\alpha})$ is solution of the following \textit{fractional Hamiltonian system}:
\begin{equation}\tag{HS${}^\alpha$}\label{eqhs}
 \left\lbrace \begin{array}{l}
 		 \cDM q = \dfrac{\partial H}{\partial w} (q,u,p,t) \\[10pt]
 		\DP p = \dfrac{\partial H}{\partial x} (q,u,p,t).
        \end{array}
\right.
\end{equation}
Finally, the introduction of these last elements allows us to prove the following theorem:
\begin{theorem}\label{thmfinal}
Let $u \in \CC^0 ([a,b],\R^m)$. Then, $u$ is a critical point of $\LL^{\alpha}$ if and only if $(q^{u,\alpha},u,p^{u,\alpha})$ is solution of the following \textit{fractional stationary equation}:
\begin{equation}\tag{SE${}^\alpha$}\label{eqse}
\dfrac{\partial H}{\partial v} (q,u,p,t) = 0.
\end{equation}
\end{theorem}

\begin{proof}
Let $u$, $\bar{u} \in \CC^0 ([a,b],\R^m)$. From Lemma \ref{lem1}, we have:
\begin{equation}
D\LL^\alpha(u)(\bar{u}) = \di \int_{a}^{b} \dfrac{\partial L}{\partial x} (q^{u,\alpha},u,t) \cdot \bar{q} + \dfrac{\partial L}{\partial v} (q^{u,\alpha},u,t) \cdot \bar{u} \; dt.
\end{equation}
Then:
\begin{multline}
D\LL^\alpha(u)(\bar{u}) = \di \int_{a}^{b} \left( \dfrac{\partial L}{\partial x} (q^{u,\alpha},u,t) + \left( \dfrac{\partial f}{\partial x} (q^{u,\alpha},u,t) \right)^T \times p^{u,\alpha} \right) \cdot \bar{q} \\
- \left( \dfrac{\partial f}{\partial x} (q^{u,\alpha},u,t) \times \bar{q} \right) \cdot p^{u,\alpha} + \dfrac{\partial L}{\partial v} (q^{u,\alpha},u,t) \cdot \bar{u} \; dt.
\end{multline}
From Theorem \ref{thmfcl}, since $\bar{q}$ is solution of \eqref{eqlcpq} and $p^{u,\alpha}$ is solution of \eqref{eqcpp}, we have:
\begin{equation}
\bar{q} = I^\alpha_- \left( \dfrac{\partial f}{\partial x} (q^{u,\alpha},u,t) \times \bar{q} + \dfrac{\partial f}{\partial v} (q^{u,\alpha},u,t) \times \bar{u} \right) 
\end{equation}
and
\begin{equation}
p^{u,\alpha} = I^\alpha_+ \left( \dfrac{\partial L}{\partial x} (q^{u,\alpha},u,t) + \left( \dfrac{\partial f}{\partial x} (q^{u,\alpha},u,t) \right)^T \times p^{u,\alpha} \right).
\end{equation}
Then, using the fractional integration by parts given in Property \ref{prop2}, we obtain:
\begin{multline}
D\LL^\alpha(u)(\bar{u}) = \di \int_{a}^{b} p^{u,\alpha} \cdot \left( \dfrac{\partial f}{\partial x} (q^{u,\alpha},u,t) \times \bar{q} + \dfrac{\partial f}{\partial v} (q^{u,\alpha},u,t) \times \bar{u} \right) \\
- \left( \dfrac{\partial f}{\partial x} (q^{u,\alpha},u,t) \times \bar{q} \right) \cdot p^{u,\alpha} + \dfrac{\partial L}{\partial v} (q^{u,\alpha},u,t) \cdot \bar{u} \; dt.
\end{multline}
Then:
\begin{eqnarray*}
D\LL^\alpha(u)(\bar{u}) & = & \di \int_{a}^{b} \left( \left( \dfrac{\partial f}{\partial v} (q^{u,\alpha},u,t) \right)^T \times p^{u,\alpha} + \dfrac{\partial L}{\partial v} (q^{u,\alpha},u,t) \right) \cdot \bar{u} \; dt \\
& = & \di \int_{a}^{b} \dfrac{\partial H}{\partial v} (q^{u,\alpha},u,p^{u,\alpha},t)\cdot \bar{u} \; dt.
\end{eqnarray*}
The proof is completed by the Dubois-Raymond's lemma.
\end{proof}

Let us note that Lemma \ref{lem1} is proved in Appendix \ref{appAd} using Lemmas \ref{lemgronf}, \ref{lemappA1} and \ref{lemappA2}. These last three Lemmas are respectively a fractional Gronwall's Lemma, a result of stability of order $1$ and a result of stability of order $2$ for the fractional Cauchy problem \eqref{eqcpq}. Hence, the proof of Theorem \ref{thmfinal} is only based on classical mathematical tools adapted to the fractional case. \\

From Theorem \ref{thmfinal}, we retrieve the following result provided in \cite{agra2,agra3,torr,torr2,jeli} and references therein:
\begin{corollary}\label{corfinal}
$\LL^\alpha$ has a critical point in $\CC^0 ([a,b],\R^m)$ if and only if there exists $(q,u,p) \in \CC^{[\alpha]} ([a,b],\R^d) \times \CC^0 ([a,b],\R^m) \times \CC^{[\alpha]} ([a,b],\R^d)$ solution of the following \textit{fractional Pontryagin's system}:
\begin{equation}\tag{PS${}^\alpha $}\label{eqps}
 \left\lbrace \begin{array}{l}
 		\cDM q = \dfrac{\partial H}{\partial w} (q,u,p,t) \\[10pt]
 		\DP p = \dfrac{\partial H}{\partial x} (q,u,p,t) \\[10pt]
	    \dfrac{\partial H}{\partial v} (q,u,p,t) = 0 \\[10pt]
	    \big( q(a),p(b) \big) = ( A,0 ).
        \end{array}
\right.
\end{equation}
In the affirmative case, $u$ is a critical point of $\LL^\alpha$ and we have $(q,p) = (q^{u,\alpha},p^{u,\alpha})$.
\end{corollary}
In practice, we are going to use more Corollary \ref{corfinal} than Theorem \ref{thmfinal}, see Examples in Section \ref{section4}. Let us note that the fractional Pontryagin's system \eqref{eqps} is made up of the fractional Hamiltonian system \eqref{eqhs}, the fractional stationary equation \eqref{eqse} and initial and final conditions. 

\subsection{Remark $1$: the classical case}\label{section23} In the case $\alpha =1$, the fractional derivatives coincide with the classical one. Consequently, the fractional optimal control problem studied coincides with the classical one. Then, in this case, Corollary \ref{corfinal} is nothing else but the classical theorem obtained in \cite{berg,favi,gunz2,gunz}:
\begin{theorem}\label{corfinalclass}
$\LL^1$ has a critical point in $\CC^0 ([a,b],\R^m)$ if and only if there exists $(q,u,p) \in \CC^{1} ([a,b],\R^d) \times \CC^0 ([a,b],\R^m) \times \CC^{1} ([a,b],\R^d)$ solution of the following \textit{Pontryagin's system}:
\begin{equation}\tag{PS${}^1 $}\label{eqpsclass}
 \left\lbrace \begin{array}{l}
 		\dot{q} = \dfrac{\partial H}{\partial w} (q,u,p,t) \\[10pt]
 		\dot{p} = - \dfrac{\partial H}{\partial x} (q,u,p,t) \\[10pt]
	    \dfrac{\partial H}{\partial v} (q,u,p,t) = 0 \\[10pt]
	    \big( q(a),p(b) \big) = (A,0).
        \end{array}
\right.
\end{equation}
In the affirmative case, $u$ is a critical point of $\LL^1$ and we have $(q,p) = (q^{u,1},p^{u,1})$.
\end{theorem}

\subsection{Remark $2$: link with the fractional Euler-Lagrange equation}\label{section24} Let us take the constraint function $f(x,v,t) = v$ satisfying \eqref{condf}. In this case, the fractional optimal control problem studied is:
\begin{equation}\label{eq24-1}
\text{optimize} \quad \di \int_a^b L(q,u,t) \; dt \quad \text{under the constraint} \quad \left\lbrace \begin{array}{l}
 		\cDM q = u \\
 		q(a) = A.
        \end{array}
\right.
\end{equation}
Using Corollary \ref{corfinal}, if this problem has a solution, then there exists a solution $(q,u,p) \in \CC^{[\alpha]} ([a,b],\R^d) \times \CC^0 ([a,b],\R^m) \times \CC^{[\alpha]} ([a,b],\R^d)$ of the fractional Pontryagin's system \eqref{eqps} here given by:
\begin{equation}
 \left\lbrace \begin{array}{l}
 		\cDM q = u \\[10pt]
 		\DP p = \dfrac{\partial L}{\partial x} (q,u,t) \\[10pt]
	    \dfrac{\partial L}{\partial v} (q,u,t) + p = 0 \\[10pt]
	    \big( q(a),p(b) \big) = (A,0).
        \end{array}
\right.
\end{equation}
In the case of the existence of an optimal control $u$ for problem \eqref{eq24-1}, we obtain that the state variable associated $q^{u,\alpha}$ is solution of the following \textit{fractional Euler-Lagrange equation}:
\begin{equation}\tag{EL${}^\alpha$}
 \dfrac{\partial L}{\partial x} (q,\cDM q,t) + \DP \left(  \dfrac{\partial L}{\partial v} (q,\cDM q,t) \right) = 0.
\end{equation}
According to the works of Agrawal in \cite{agra}, we then obtain that $q^{u,\alpha}$ is a critical point of the following \textit{fractional Lagrangian functional}:
\begin{equation}
q \longrightarrow \di \int_a^b L(q,\cDM q,t) \; dt.
\end{equation}
We refer to \cite{agra} for more details concerning fractional Euler-Lagrange equations.

\section{A transition from the classical to the fractional problem}\label{section3}
As seen in Introduction, there is a large development in fractional calculus and consequently concerning optimal control problems with controlled systems coded by fractional differential equations. In this section, we present a link between a classical optimal control problem and its fractional version via the \textit{Stanislavsky's formalism}, see \cite{iniz,stan}. Only for this section, we assume that $[a,b]=[0,\tau] $ with $ \tau > 0 $  and $ 0 < \alpha < 1$. 

\subsection{The Stanislavsky's formalism}\label{section31}
The governing equation for a fluid in a homogeneous porous medium is usually the classical Richards equation and it is derived from physical principles, see \cite{rich}. Nevertheless, the fractional Richards equation plays an important role in the study of the behaviour of a fluid in a heterogeneous medium, \cite{hilf,neel,neel2}. A natural question then arises: \textit{how physically understand the emergence of the fractional order when the medium changes?} Let us sketch an example of answer. A heterogeneous porous medium presents a lot of stochastic heterogeneities implying complex geometric structure and then inducing stochastic retention zones, \cite{neel}. Hence, we can assume that a change of medium induces a \textit{modification of the time variable}. Therefore, following the work of Stanislavsky consisting in the introduction of a stochastic internal time (a "slow" time), we remind that the emergence of a fractional derivative can be the effect of a change of time. Let us give more details. However, we refer to \cite{iniz,stan} for a complete study. \\

Let us consider $T_t$ a stochastic process of probability density function $\rho (y,t)$. $T_t$ represents the stochastic internal time (the "slow" time). Stanislavsky assumes that the Laplace transform $ \text{Lap}_{(y)}$ of $\rho (y,t)$ with respect to its first variable satisfies:
\begin{equation}
\forall t \geq 0, \; \forall s \in \R,\; \text{Lap}_{(y)} [\rho (y,t)](s)={\rm E}_{\alpha,1}(-st^\alpha),
\end{equation}
where ${\rm E}_{\alpha,1}$ is the Mittag-Leffler function defined in Appendix \ref{appAa}. A possible construction of $T_t$ is given in \cite{stan}. \\

Stanislavsky studies the dynamical effects of this change of time. Precisely, for any $g \in \CC^0([0,\tau],\R^d)$ and any $t \in [0,\tau]$, he studies the function $\mathcal{F}_\alpha (g)$ defined by:
\begin{equation}
\mathcal{F}_\alpha (g)(t) = \mathbb{E} \big( g(T_t) \big),
\end{equation}
where $\mathbb{E}$ designates the mean value. The main property is the following result proved in \cite{iniz,stan}:
\begin{lemma}\label{lemstan}
Let $g \in \CC^1([0,\tau],\R^d)$. Then, we have:
\begin{equation}
\forall t \in [0,\tau], \; \mathbb{E} \big( \dot{g} (T_t) \big) = \cDM \big( \mathcal{F}_\alpha (g) \big) (t).
\end{equation}
\end{lemma}
Hence, a classical derivative is transformed into a fractional one under the introduction of the stochastic internal time of Stanislavsky. 

\subsection{Application on the class of fractional optimal control problems}\label{section32}
Now, let us talk about the consequences of the works of Stanislavsky on the class of fractional optimal control problems studied in this paper. We use the notations and definitions of Section \ref{section22}. \\

Let us assume that $L$ and $f$ are autonomous, \ie $L(x,v,t) = L(x,v)$ and $f(x,v,t) = f(x,v)$. Then, let us assume that $f$ satisfies the following condition:
\begin{multline}\label{condfstan}
\forall (q,u) \in \CC^1 ([0,\tau],\R^d) \times \CC^0 ([0,\tau],\R^m), \; \forall t \in [0,\tau], \\
\mathbb{E}\Big( f \big( q(T_t),u(T_t) \big) \Big) = f \Big( \mathbb{E} \big( q(T_t) \big),\mathbb{E}\big( u(T_t) \big) \Big) .
\end{multline}
Such a condition is satisfied by linear autonomous constraint functions, as the examples in Section \ref{section4}. According to Lemma \ref{lemstan}, one can prove that for any control $u$, $\mathcal{F}_\alpha (q^{u,1})$ satisfies:
\begin{equation}
\left\lbrace \begin{array}{l}
\cDM \mathcal{F}_\alpha (q^{u,1}) = f \big( \mathcal{F}_\alpha (q^{u,1}),\mathcal{F}_\alpha (u) \big) \\
\mathcal{F}_\alpha (q^{u,1}) (a) =A.
\end{array} \right.
\end{equation}
Precisely, we have for any control $u$, $ \mathcal{F}_\alpha (q^{u,1}) = q^{\mathcal{F}_\alpha (u),\alpha} $. \\

Now, let us consider the optimal control problem studied in Section \ref{section22} in the classical case: 
\begin{equation}
\text{optimize the cost functional} \quad u \in \CC^0 ([a,b],\R^d) \longmapsto \di \int_0^\tau L(q^{u,1},u,t) \; dt.
\end{equation}
Let us assume that the controlled system can be prone to changes inducing a modification of the time variable as it is the case with porous media and the Richards equation. We are then interested in the optimization of the following "slow" cost functional:
\begin{equation}\label{eq32-1}
u \in \CC^0 ([a,b],\R^d) \longmapsto \di \int_0^\tau L\big( \mathcal{F}_\alpha (q^{u,1}),\mathcal{F}_\alpha (u) \big) \; dt
\end{equation}
which is relied to the following \textit{fractional} cost functional:
\begin{equation}\label{eq32-2}
u \in \CC^0 ([a,b],\R^d) \longmapsto \di \int_0^\tau L ( q^{u,\alpha},u ) \; dt 
\end{equation}
as its restriction to the image of $\mathcal{F}_\alpha$. Hence, once an optimal control of the functional \eqref{eq32-2} found, one can be interested in its projection on the image of $\mathcal{F}_\alpha$ in order to approach an optimal control of the functional \eqref{eq32-1}. \\

Hence, the Stanislavsky's formalism is an example relying a classical optimal control problem to its fractional version via the introduction of a stochastic internal time.

\section{Examples}\label{section4} 
In this section, we are going to study some examples of optimal control problems studied in Section \ref{section22} both in classical and fractional cases.
\subsection{An applied classical linear-quadratic example}\label{section41}
Classical linear-quadratic examples are often studied in the literature because they are used for tracking problems. The aim of these problems is to determine a control allowing to approach as much as possible reference trajectories, \cite[Part 1.4, p.49]{trel}. In this section, we study such an example, \cite[Part 4.4.3, example 3, p.53]{evan}. More generally, a quadratic Lagrangian is often natural (for example in order to minimize distances) and even if the differential equations are frequently non linear, we are often leaded to study linearised versions. \\

In this section, we take $\alpha = d = m = 1$. We consider an evolution problem of two nondescript populations $z$ and $\xi$ interacting during a time interval $[a,b]$. We control the injection or the discharge of the population $\xi$ in the system: we then control the value of the population $\xi$ in real time. The interaction between $z$ and $\xi$ is governed by the following linear differential equation:
\begin{equation}\label{eq41-000}
\dot{z} = z + \xi .
\end{equation}
For any control $\xi$ and any real $z_0$, we denote by $z_{z_0,\xi}$ the unique solution of \eqref{eq41-000} associated to $\xi$ and satisfying the initial condition $z(a)=z_0 \in \R$. Let us assume that, for an initial condition $z(a)=X \in \R$, we know experimentally a satisfactory control $\xi_X$ such that the evolution of the population $z_{X,\xi_X}$ is healthy with respect to a nondescript constraint. \\

The problem is finally the following: \textit{assume that the initial condition is modified (\textit{i.e.} $z(a)=Y \neq X$), what is the control $\xi$ minimizing the difference between $\xi$ and $\xi_X$ plus the difference between $z_{Y,\xi}$ and $z_{X,\xi_X}$?} More precisely, we are looking for a minimizer of the following functional:
\begin{equation}\label{eq41-0}
\xi \in \CC^0([a,b],\R) \longmapsto \dfrac{1}{2} \di \int_a^b (\xi - \xi_X)^2 + (z_{Y,\xi} -z_{X,\xi_X})^2 \; dt.
\end{equation}
With a change of variable $A=Y-X$, $u=\xi - \xi_X$, $q=z_{Y,\xi} -z_{X,\xi_X}$ and giving the following quadratic Lagrangian and the following linear constraint function:
\begin{equation}\label{eq41-00}
\fonction{L}{\R^2 \times [a,b]}{\R}{(x,v,t)}{(x^2+v^2)/2} \quad \text{and} \quad \fonction{f}{\R^2 \times [a,b]}{\R}{(x,v,t)}{x+v,}
\end{equation}
the problem is reduced to find an optimal control $u$ for the cost functional $\LL^1$ associated to $L$, $f$ and $A$. See Section \ref{section22} for notations and definitions. \\

According to Theorem \ref{corfinalclass}, we are interested in solving the Pontryagin's system \eqref{eqpsclass} here given by:
\begin{equation}
 \left\lbrace \begin{array}{l}
 		\dot{q} = q+u \\
 		\dot{p} = - q-p \\
	    p+u = 0 \\
	    \big( q(a),p(b) \big) = ( A,0 ).
        \end{array}
\right.
\end{equation}
Hence, with $p+u = 0$, it is sufficient to solve the following linear Cauchy problem:
\begin{equation}\label{eq41-1}
 \left\lbrace \begin{array}{l}
 		\dot{q} = q+u \\
 		\dot{u} = q-u \\
	    \big( q(a),u(b) \big) = ( A,0 ).
        \end{array}
\right.
\end{equation}
Using matrix exponentials, one can prove that it exists an unique solution $(q,u)$ to Cauchy problem \eqref{eq41-1} given by:
\begin{equation}
\forall t \in [a,b], \; \left\lbrace \begin{array}{l}
 		q(t) = A \Big[ \cosh \big( \sqrt{2}(t-a)\big)+\dfrac{1-R}{\sqrt{2}} \sinh \big( \sqrt{2}(t-a)\big) \Big] \\[10pt]
 		u(t) = A \Big[\dfrac{1+R}{\sqrt{2}} \sinh \big( \sqrt{2}(t-a)\big)- R \cosh \big( \sqrt{2}(t-a)\big) \Big]
        \end{array}
\right.
\end{equation}
where
\begin{equation}
R=\dfrac{\sinh \big( \sqrt{2}(b-a)\big)}{\sqrt{2}\cosh \big( \sqrt{2}(b-a)\big)-\sinh \big( \sqrt{2}(b-a)\big)} .
\end{equation}
Finally, we conclude that $u$ is the unique critical point of $\LL^1$. Hence, if there exists $\xi$ minimizing \eqref{eq41-0}, then:
\begin{equation}
\forall t \in [a,b], \; \xi(t) = \xi_X (t) + u(t).
\end{equation}

\subsection{The fractional linear quadratic example}\label{section42}
In this section, we consider the previous example of Section \ref{section41} in the strict fractional case $0<\alpha <1$. Let us consider $d=m=1$ and the couple $(L,f)$ given in Equation \eqref{eq41-00}. We want to express a critical point $u$ of the cost functional $\LL^\alpha$ associated to $L$ and $f$. See Section \ref{section22} for notations and definitions. \\

Therefore, according to Corollary \ref{corfinal}, we look for $(q,u,p)$ solution of the fractional Pontryagin's system \eqref{eqps} here given by:
\begin{equation}\label{eq42-0}
 \left\lbrace \begin{array}{l}
 		\cDM q = q+u \\
 		\DP p = q+p \\
	    p+u = 0 \\
	    \big( q(a),p(b) \big) = ( A,0 ).
        \end{array}
\right.
\end{equation}
Hence, from $u+p=0$, it is sufficient to solve the following linear fractional Cauchy problem:
\begin{equation}\label{eq42-1}
 \left\lbrace \begin{array}{l}
 		\cDM q = q+u \\
 		\DP u = u-q \\
	    \big( q(a),u(b) \big) = ( A,0 ).
        \end{array}
\right.
\end{equation}
Although the fractional Cauchy problem \eqref{eq42-1} is linear, the presence of the left and the right fractional derivatives is a main drawback to the explicit computation of a solution. More generally, this characteristic implies very big difficulties in order to solve the most of fractional Hamiltonian systems. Finally, the resolution of \eqref{eqps} is the most of time still an opened problem. \\

Nevertheless, in this case, we can hope getting informations on the solutions of \eqref{eqps} in dimension $d=m=2$ with the help of the fractional Noether's theorem proved in Section \ref{section5}. Indeed, in this case, the fractional Pontryagin's system \eqref{eqps} admits a symmetry and a constant of motion can be obtained. We refer to Section \ref{section5} for more details.

\subsection{A solved fractional example}\label{section43} 
In this section, we give a solved fractional example in the sense that we give explicitly the unique critical point of the cost functional $\LL^\alpha$. Nevertheless, the fractional Pontryagin's system \eqref{eqps} is still not completely resolved. \\

Let us take $d=m=1$, $[a,b]=[0,1]$ and $0<\alpha \leq 1$. Let us consider the following Lagrangian $L$ and the following linear constraint function $f$:
\begin{equation}\label{eq43-1}
\fonction{L}{\R^2 \times [0,1]}{\R}{(x,v,t)}{(v^2/2) + \gamma (1-t)^\beta x} \quad \text{and} \quad \fonction{f}{\R^2 \times [0,1]}{\R}{(x,v,t)}{\lambda x + \mu v,}
\end{equation}
where $\beta$, $\gamma$, $\mu$, $\lambda$ are parameters in $\R^*$. We are looking for a critical point $u$ of $\LL^\alpha$ associated to $L$ and $f$. See Section \ref{section22} for notations and definitions. \\

According to Corollary \ref{corfinal}, we have to solve the fractional Pontryagin's system \eqref{eqps} here given by:
\begin{equation}
 \left\lbrace \begin{array}{l}
 		\cDM q = \lambda q + \mu u \\
 		\DP p = \lambda p + \gamma (1-t)^\beta \\
	    u + \mu p = 0 \\
	    \big( q(0),p(1) \big) = ( A,0 ).
        \end{array}
\right.
\end{equation}
From Theorem \ref{thmfcl}, the Cauchy problem $\cDM q = \lambda q + \mu u $ with $q(0)=A $ admits an unique solution for any $u \in \CC^0([0,1],\R)$. Then, from $u + \mu p = 0$, it is sufficient to solve the following fractional Cauchy problem:
\begin{equation}
 \left\lbrace \begin{array}{l}
 		\DP p = \lambda p + \gamma (1-t)^\beta \\
	    p(1) = 0
        \end{array}
\right. \quad \text{equivalent to} \quad \left\lbrace \begin{array}{l}
 		\DM p_0 = \lambda p_0 + \gamma t^\beta \\
	    p_0(0) = 0,
        \end{array}
\right.
\end{equation}
with the change of unknown $p_0(t) = p(1-t) $ for any $t \in [0,1]$. The unique solution of this last fractional Cauchy problem is given in \cite[Chap.3, p.137]{kilb} by:
\begin{equation}
\forall t \in [0,1], \;  p_0 (t) = \di \int_0^t (t-y)^{\alpha-1} {\rm E}_{\alpha,\alpha} \big( \lambda (t-y)^\alpha \big) \gamma y^\beta \; dy,
\end{equation}
where ${\rm E}_{\alpha,\alpha}$ is the Mittag-Leffler function defined in Appendix \ref{appAa}. In order to get a better formulation, we make a change of variable which gives us:
\begin{eqnarray}
\forall t \in [0,1], \;  p_0 (t) & = & \gamma \di \int_0^t (t-y)^{\beta} y^{\alpha-1} {\rm E}_{\alpha,\alpha} ( \lambda y^\alpha ) dy \\
& = & \gamma \Gamma (\beta +1) I^{\beta +1}_- \big( y^{\alpha-1} {\rm E}_{\alpha,\alpha} ( \lambda y^\alpha ) \big) (t) \\[5pt]
& = & \gamma \Gamma (\beta +1) t^{\alpha+\beta} {\rm E}_{\alpha,\alpha+\beta+1} ( \lambda t^\alpha ).
\end{eqnarray}
We refer to \cite{kilb} for more details concerning the calculations. Finally, we obtain the unique critical point $u$ of $\LL^\alpha$ given by: 
\begin{equation}
\forall t \in [0,1], \; u(t) = - \mu \gamma \Gamma (\beta +1) (1-t)^{\alpha+\beta} {\rm E}_{\alpha,\alpha+\beta+1} \big( \lambda (1-t)^\alpha \big) .
\end{equation}
Although we obtain the unique critical point $u$ of $\LL^\alpha$, let us note that this example does not provide a completely solved fractional Pontryagin's system: the state variable $q^{u,\alpha}$ is still unknown. However, this example allows to test the quality of numerical schemes giving approximations of critical points of cost functionals $\LL^\alpha$. This will be done in a further paper for a variational integrator of \eqref{eqps}.

\section{A fractional Noether's theorem}\label{section5}
As we have seen with concrete examples in Sections \ref{section42} and \ref{section43}, fractional Pontryagin's systems \eqref{eqps} are the most of time not resolvable explicitly. This is a strong obstruction in order to express explicitly a critical point of a cost functional $\LL^\alpha$. In this section, we suggest a deviously way in order to get informations on the solutions of \eqref{eqps} and consequently on the critical points of the cost functional $\LL^\alpha$ associated. Precisely, we are interested in the existence of conservation laws for fractional Pontryagin's systems admitting a symmetry (see Definition \ref{defsym}). \\

Indeed, in 1918, Noether proved the existence of an explicit conservation law for any classical Euler-Lagrange equation admitting a symmetry. We refer to \cite{arno,noet2,noet} for more details. Adapting her strategy to the fractional case, Cresson, Torres and Frederico proved in 2007 a preliminary result giving a conservation law for any fractional Euler-Lagrange equation admitting a symmetry, see \cite{cres6,torr3,torr4}. Nevertheless, the conservation law obtained was not explicit but only given implicitly via a functional relation. From this first result, we have formulated in \cite{bour2} a fractional Noether's theorem providing an \textit{explicit} formulation of this conservation law via a \textit{transfer formula}. \\

At the same time, in \cite{torr,torr2}, Torres and Frederico applied a similar strategy for fractional Pontryagin's system admitting a symmetry. Nevertheless, this conservation law is also given implicitly via a functional relation. In this section, we are going to formulate a fractional Noether's theorem providing an \textit{explicit} formulation of this conservation law via an other \textit{transfer formula}. \\

We first review the definition of a one parameter group of diffeomorphisms: 
\begin{definition}
Let $n \in \N^*$. For any real $s$, let $\fonctionsansdef{\phi (s,\cdot)}{\R ^n}{\R ^n}$ be a diffeomorphism. Then, $\Phi = \{ \phi (s,\cdot) \}_{s \in \R}$ is a one parameter group of diffeomorphisms of $\R^n$ if it satisfies:
\begin{enumerate}
\item $\phi (0,\cdot) = Id_{\R ^n}$; 
\item $\forall s,s' \in \R, \; \phi (s,\cdot) \circ \phi (s',\cdot) = \phi (s+s',\cdot) $;
\item $\phi$ is of class $\CC^2$.
\end{enumerate}
\end{definition}

Usual examples of one parameter groups of diffeomorphisms are given by translations and rotations. The action of three one parameter groups of diffeomorphisms on an Hamiltonian allows to define the notion of a symmetry for a fractional Pontryagin's system:

\begin{definition}\label{defsym}
For $i=1,2,3$, let $\Phi_i = \{ \phi_i (s,\cdot) \}_{s \in \R}$ be a one parameter group of diffeomorphisms of $\R^d$, $ \R^m$ and $\R^d$ respectively. Let $L$ be a Lagrangian, $f$ be a constraint function and $H$ be the associated Hamiltonian. $H$ is said to be $\cDM$-invariant under the action of $(\Phi_i)_{i=1,2,3}$ if it satisfies for any $(q,u,p)$ solution of \eqref{eqps} and any $s \in \R$:
\begin{equation}
H \Big( \phi_1 (s,q), \phi_2 (s,u) , \phi_3 (s,p) ,t \Big) - \phi_3 (s,p) \cdot \cDM \big( \phi_1 (s,q) \big) = H(q,u,p,t) - p \cdot \cDM q.
\end{equation}
\end{definition}

From this notion, Torres and Frederico proved in \cite{torr,torr2} the following result:

\begin{lemma}\label{lemtorr}
Let $L$ be a Lagrangian, $f$ be a constraint function and $H$ be the associated Hamiltonian. Let us assume that $H$ is $\cDM$-invariant under the action of three one parameter groups of diffeomorphisms $(\Phi_i)_{i=1,2,3}$. Then, the following equality holds for any solution $(q,u,p)$ of \eqref{eqps}:
\begin{equation}\label{eqlemtorr}
\cDM \left( \dfrac{\partial \phi_1}{\partial s} (0,q) \right) \cdot p - \dfrac{\partial \phi_1}{\partial s} (0,q) \cdot \DP p = 0.
\end{equation}
\end{lemma}

Then, from this result, Torres and Frederico defined a notion of \textit{fractional-conserved quantity}. Nevertheless, this result did not provide exactly a constant of motion. In this section, using a transfer formula, we are going to write the left term of Equation \eqref{eqlemtorr} as an explicit classical derivative and then we obtain a real constant of motion. Let us provide this transfer formula:

\begin{lemma}[Transfer formula]\label{lemtransfer}
Let $g_1$, $g_2 \in \CC^{\infty} ([a,b],\R^d)$ satisfying the following condition {\rm (C)}:
\begin{center}
the sequences of functions $ \big( I^{p-\alpha}_- \big(g_1-g_1(a)\big) \cdot g_2^{(p)} \big)_{p \in \N^*} $ and \\ $ ( g_1^{(p)} \cdot I^{p-\alpha}_+ g_2 )_{p \in \N^*} $ converge uniformly to $0$ on $ [a,b] $.
\end{center}
Then, the following equality holds:
\begin{equation}\label{eqlemtransfer}
\cDM g_1 \cdot g_2 - g_1 \cdot \DP g_2 = \dfrac{d}{dt} \left[ \di \sum_{r=0}^{\infty} (-1)^r I^{r+1-\alpha}_- \big(g_1-g_1(a)\big) \cdot g_2^{(r)} + g_1^{(r)} \cdot I^{r+1-\alpha}_+ g _2\right]. 
\end{equation}
\end{lemma}

\begin{proof}
In \cite{bour2}, under a similar condition than Condition {\rm (C)}, we have proved:
\begin{equation}\label{}
\DM g_1 \cdot g_2 = \dfrac{d}{dt} \left[ \di \sum_{r=0}^{\infty} (-1)^r I^{r+1-\alpha}_- g_1 \cdot g_2^{(r)} \right] \; \text{and} \; - g_1 \cdot \DP g_2 = \dfrac{d}{dt} \left[ \di \sum_{r=0}^{\infty} g_1^{(r)} \cdot I^{r+1-\alpha}_+ g _2\right]. 
\end{equation}
Consequently, writing $\cDM g_1 = \DM \big( g_1 - g_1(a) \big)$, the proof is completed. 
\end{proof}

A discussion is provided in \cite{bour2} concerning the condition (C): one can prove that this condition is satisfied for any couple of analytic functions for example. Then, combining Lemmas \ref{lemtorr} and \ref{lemtransfer}, we prove:

\begin{theorem}[A fractional Noether's theorem]\label{thmnoetherfrac}
Let $L$ be a Lagrangian, $f$ be a constraint function and $H$ be the associated Hamiltonian. Let us assume that $H$ is $\cDM$-invariant under the action of three one parameter groups of diffeomorphisms $(\Phi_i)_{i=1,2,3}$. Let $(q,u,p)$ be a solution of \eqref{eqps} and let $g$ denote $\partial \phi_1 / \partial s( 0,q )$. If $g$ and $p$ satisfy Condition {\rm (C)}, then the following equality holds:
\begin{equation}\label{eqthmnoetherfrac}
\dfrac{d}{dt} \left[ \di \sum_{r=0}^{\infty} (-1)^r I^{r+1-\alpha}_- \big(g-g(a)\big) \cdot p^{(r)} + g^{(r)} \cdot I^{r+1-\alpha}_+ p \right] = 0 .
\end{equation}
\end{theorem}

This theorem then provides an explicit algorithmic way to compute a constant of motion for any fractional Pontryagin's systems admitting a symmetry. Nevertheless, the conservation law is only given as a series of functions: in most cases, it is not easily computable. However, an arbitrary closed approximation of this quantity can be obtained with a truncation. \\

Let us note that the fractional linear-quadratic example developed in Section \ref{section42} gives a concrete example of fractional Pontryagin's system admitting a symmetry:
\begin{example}\label{ex4}
Let us consider $d = m = 2$, the following quadratic Lagrangian and the following linear constraint function
\begin{equation}
\fonction{L}{\R^2 \times \R^2 \times [0,1]}{\R}{(x,v,t)}{(\Vert x \Vert^2 + \Vert v \Vert^2 )/2} \quad \text{and} \quad \fonction{f}{\R^2 \times \R^2 \times [0,1]}{\R^2}{(x,v,t)}{x+v.}
\end{equation}
Then, let us consider the three one parameter groups of diffeomorphisms given by the following rotations:
\begin{equation}\label{rotation}
\fonction{\phi_i}{\R \times \R^2}{\R^2}{(s,x_1,x_2)}{\left( \begin{array}{cc} \cos (s \theta_i) & - \sin (s \theta_i) \\ \sin (s \theta_i) & \cos (s \theta_i)  \end{array} \right) \left( \begin{array}{c} x_1 \\ x_2 \end{array} \right),}
\end{equation}
for $i=1,2,3$ and where $\theta_1$, $\theta_2 \in \R$ and $\theta_3 = - \theta_1$. With these parameters, one can prove that the Hamiltonian $H$ associated to $L$ and $f$ is $\cDM$-invariant under the action of $(\Phi_i)_{i=1,2,3}$. Consequently, the fractional Pontryagin's system given in \eqref{eq42-0} (in dimension $2$) is not resolvable but admits a symmetry and then admits an explicit conservation law given by the fractional Noether's Theorem \ref{thmnoetherfrac}. As said previously, this constant of motion is not explicitly computable. However, truncating the series, one can provide an approximation of this quantity.
\end{example}

\appendix

\section{Appendix}\label{appA}

\subsection{A fractional Gronwall's lemma}\label{appAa}
Let us recall the definition extracted from \cite{kilb,podl,samk} of the Mittag-Leffler function ${\rm E}_{\alpha_1,\alpha_2}$ of parameters $\alpha_1$, $\alpha_2 \geq 0$:
\begin{equation}
\forall t \in \R, \; {\rm E}_{\alpha_1,\alpha_2} (t) := \di \sum_{k=0}^{\infty} \dfrac{t^k}{\Gamma (\alpha_1 k + \alpha_2)}.
\end{equation}
Now, let us give the following Lemma proved in \cite{diet,gao}. For the reader's convenience, we recall the proof:
\begin{lemma}[A fractional Gronwall's lemma]\label{lemgronf}
Let $g \in \CC^0 ([a,b],\R)$ and $\alpha >0$. Let us assume that there exist $K_1$, $K_2 \geq 0 $ such that:
\begin{equation}\label{eqAa-1}
\forall t \in [a,b], \; 0 \leq g(t) \leq K_1 I^\alpha_- g(t) + K_2.
\end{equation}
Then, $g$ satisfies:
\begin{equation}\label{eqAa-2}
\forall t \in [a,b],\;  0 \leq g(t) \leq K_2 {\rm E}_{\alpha,1} \big( K_1 (t-a)^\alpha \big).
\end{equation}
\end{lemma}

\begin{proof}
Using Property \ref{prop1}, Assumption \eqref{eqAa-1} implies by induction that for any $n \in \N^*$ and any $t \in [a,b]$:
\begin{equation}\label{eqAa-3}
0 \leq g(t) \leq K_1^n  I^{n \alpha}_- g(t) + K_2 \di \sum_{k=0}^{n-1} K_1^k I^{k\alpha}_- (1) (t) = K_1^n  I^{n \alpha}_- g(t) + K_2 \di \sum_{k=0}^{n-1} \dfrac{\big( K_1(t-a)^{\alpha}\big)^k}{\Gamma (\alpha k+1)}.
\end{equation}
Since $g$ is continuous on $[a,b]$ and non-negative, $g$ is bounded by $0$ and by a constant $K_3 \geq 0$. Hence, using the Stirling asymptotic formula, we have for any $t \in [a,b]$:
\begin{equation}
0 \leq K_1^n I^{n \alpha}_- g(t) = \dfrac{K_1^n}{\Gamma (\alpha n)} \di \int_a^t (t-y)^{n\alpha-1} g(y) dy  \leq  K_3 \dfrac{\big( K_1 (t-a)^{\alpha} \big)^n}{\Gamma (\alpha n+1)} \underset{n\to+\infty}{\longrightarrow}  0.
\end{equation}
Finally, making $n$ tend to $\infty$ in Inequality \eqref{eqAa-3}, the proof is completed.
\end{proof}
Let us note that ${\rm E}_{1,1}$ is the exponential function. Consequently, in the classical case $\alpha =1$, the fractional Gronwall's lemma is nothing else but the classical Gronwall's lemma.

\subsection{Result of stability of order $1$}\label{appAb}
In this section, we use the notations and definitions given in Section \ref{section22}. We prove the following Lemma with the help of Lemma \ref{lemgronf}.
\begin{lemma}\label{lemappA1}
Let $u$, $\bar{u} \in \CC^0 ([a,b],\R^m)$. Then, there exists a constant $C_1 \geq 0$ such that:
\begin{equation}
\forall \vert \eps \vert < 1, \; \forall t \in [a,b], \; \Vert q^{u+ \eps \bar{u},\alpha} (t) - q^{u,\alpha}(t) \Vert \leq C_1 \vert \eps \vert .
\end{equation}
\end{lemma}
\begin{proof}
From Theorem \ref{thmfcl}, we have for any $\abs{\eps} < 1$ and any $t \in [a,b]$:
\begin{equation}
q^{u,\alpha} (t) = A + I^{\alpha}_- \big( f( q^{u,\alpha},u,t ) \big) (t) \quad \text{and} \quad q^{u+ \eps \bar{u},\alpha} (t) = A + I^{\alpha}_- \big( f( q^{u+ \eps \bar{u},\alpha},u+ \eps \bar{u},t ) \big) (t).
\end{equation}
Hence, for any $\abs{\eps} < 1$ and any $t \in [a,b]$, we have:
\begin{eqnarray}
\Vert q^{u+ \eps \bar{u},\alpha} (t) -  q^{u,\alpha} (t) \Vert & = & \Vert I^{\alpha}_- \big( f( q^{u+ \eps \bar{u},\alpha},u+ \eps \bar{u},t ) - f( q^{u,\alpha},u,t ) \big) (t) \Vert  \\ \label{eqAb-1}
& \leq & I^{\alpha}_- \Big( \Vert f( q^{u+ \eps \bar{u},\alpha},u+ \eps \bar{u},t ) - f( q^{u,\alpha},u,t ) \Vert \Big) (t) . \label{eqAb-2}
\end{eqnarray}
We have for any $\abs{\eps} < 1$ and any $y \in [a,b]$:
\begin{multline}
\Vert f\big( q^{u+ \eps \bar{u},\alpha}(y),u(y)+ \eps \bar{u}(y),y \big) - f\big( q^{u,\alpha}(y),u(y),y \big) \Vert \\ \leq \Vert f\big( q^{u+ \eps \bar{u},\alpha}(y),u(y)+ \eps \bar{u}(y),y \big) - f\big( q^{u,\alpha}(y),u(y)+ \eps \bar{u}(y),y \big) \Vert \\ +  \Vert f\big( q^{u,\alpha}(y),u(y)+ \eps \bar{u}(y),y \big) - f\big( q^{u,\alpha}(y),u(y),y \big) \Vert.
\end{multline}
From Condition \eqref{condf} and with a Taylor's expansion of order $1$ with explicit remainder, the following inequality holds for any $\abs{\eps} < 1$ and any $y \in [a,b]$:
\begin{multline}\label{eqAb-3}
\Vert f\big( q^{u+ \eps \bar{u},\alpha}(y),u(y)+ \eps \bar{u}(y),y \big) - f\big( q^{u,\alpha}(y),u(y),y \big) \Vert \\ \leq M \Vert q^{u+ \eps \bar{u},\alpha}(y) - q^{u,\alpha}(y)\Vert + \vert \eps \vert \left\Vert \dfrac{\partial f}{\partial v} \big( q^{u,\alpha}(y),\xi^{\eps}(y),y \big) \times \bar{u}(y) \right\Vert.
\end{multline}
where $\xi^{\eps}(y) \in [u(y),u(y)+ \eps \bar{u}(y)] \subset [-M_1,M_1]^m$ with $M_1 \geq 0$ independent of $\vert \eps \vert < 1$ since $u$ and $\bar{u}$ are continuous on $[a,b]$. \\

Thus, since $\partial f / \partial v$, $q^{u,\alpha}$ and $\bar{u}$ are continuous, there exists $M_2 \geq 0$ such that for any $\vert \eps \vert < 1$, we have:
\begin{multline}\label{eqAb-4}
\forall y \in [a,b], \; \left\Vert \dfrac{\partial f}{\partial v} \big( q^{u,\alpha}(y),\xi^{\eps}(y),y \big) \times \bar{u}(y) \right\Vert \leq M_2 \\ \text{and then, } \; \forall t \in [a,b], \; I^{\alpha}_- \left( \left\Vert \dfrac{\partial f}{\partial v} ( q^{u,\alpha},\xi^{\eps},t ) \times \bar{u} \right\Vert \right) (t) \leq M_3,
\end{multline}
where $M_3 := M_2 (b-a)^\alpha / \Gamma ( \alpha +1)$ is independent of $\vert \eps \vert < 1$. \\

Finally, from Inequalities \eqref{eqAb-2}, \eqref{eqAb-3} and \eqref{eqAb-4}, we have:
\begin{equation}
\forall \vert \eps \vert < 1, \forall t \in [a,b], \; \Vert q^{u+ \eps \bar{u},\alpha} (t) -  q^{u,\alpha} (t) \Vert \leq M  I^{\alpha}_- \Big( \Vert q^{u+ \eps \bar{u},\alpha} - q^{u,\alpha}\Vert \Big) (t) +  M_3 \vert \eps \vert.
\end{equation}
Using the fractional Gronwall's Lemma \ref{lemgronf}, we conclude that:
\begin{eqnarray}
\forall \vert \eps \vert < 1, \; \forall t \in [a,b], \; \Vert q^{u+ \eps \bar{u},\alpha} (t) -  q^{u,\alpha} (t) \Vert & \leq & M_3 \vert \eps \vert {\rm E}_{\alpha,1} \big( M (t-a)^\alpha \big) \\[10pt]
& \leq & M_3 \vert \eps \vert {\rm E}_{\alpha,1} \big( M (b-a)^\alpha \big).
\end{eqnarray}
Defining $C_1 := M_3 {\rm E}_{\alpha,1} \big( M (b-a)^\alpha \big)$, the proof is completed.
\end{proof}

\subsection{Result of stability of order $2$}\label{appAc}
In this section, we use the notations and definitions given in Section \ref{section22}. We prove the following Lemma with the help of Lemmas \ref{lemgronf} and \ref{lemappA1}.
\begin{lemma}\label{lemappA2}
Let $u$, $\bar{u} \in \CC^0 ([a,b],\R^m)$. There exists a constant $C \geq 0$ such that:
\begin{equation}
\forall \vert \eps \vert < 1, \; \forall t \in [a,b], \; \Vert q^{u+ \eps \bar{u},\alpha} (t) - q^{u,\alpha}(t) - \eps \bar{q}(t) \Vert \leq C \eps^2 .
\end{equation}
where $\bar{q}$ is the unique global solution of the following linearised Cauchy's problem:
\begin{equation}\tag{LCP${}^\alpha_{\bar{q}}$}
 \left\lbrace \begin{array}{l}
 		\cDM \bar{q} = \dfrac{\partial f}{\partial x} (q^{u,\alpha},u,t) \times \bar{q} + \dfrac{\partial f}{\partial v} (q^{u,\alpha},u,t) \times \bar{u} \\[10pt]
 		\bar{q}(a) = 0.
        \end{array}
\right. 
\end{equation}
The existence and the uniqueness of $\bar{q}$ are given by Theorem \ref{thmfcl} and Condition \eqref{condf}.
\end{lemma}
\begin{proof}
We proceed in the same manner than for Lemma \ref{lemappA1}. We have for any $\abs{\eps} < 1$ and any $t \in [a,b]$:
\begin{multline}\label{eq654}
\Vert q^{u+\eps \bar{u},\alpha}(t) - q^{u,\alpha}(t) - \eps \bar{q}(t) \Vert \\ \leq I^\alpha_- \left( \left\Vert f(q^{u+\eps \bar{u},\alpha},u+\eps \bar{u},t)-f(q^{u,\alpha},u,t) - \eps \dfrac{\partial f}{\partial x}(q^{u,\alpha},u,t) \times \bar{q} - \eps \dfrac{\partial f}{\partial v}(q^{u,\alpha},u,t) \times \bar{u} \right\Vert \right)(t).
\end{multline}
With a Taylor's expansion's of order $2$ with explicit remainder, we have for any $\abs{\eps} < 1$ and any $y \in [a,b]$:
\begin{multline}\label{eq987}
\left\Vert f\big(q^{u+\eps \bar{u},\alpha}(y),u(y)+\eps \bar{u}(y),y\big)-f\big(q^{u,\alpha}(y),u(y),y\big) \right. \\ \left. - \eps \dfrac{\partial f}{\partial x}\big(q^{u,\alpha}(y),u(y),y\big) \times \bar{q}(y) - \eps \dfrac{\partial f}{\partial v}\big(q^{u,\alpha}(y),u(y),y\big) \times \bar{u}(y) \right\Vert  \\ \leq \left\Vert \dfrac{\partial f}{\partial x}\big(q^{u,\alpha}(y),u(y),y\big) \times \big(q^{u+\eps \bar{u},\alpha}(y)- q^{u,\alpha}(y)-\eps \bar{q}(y) \big) \right\Vert \\ + \left\Vert \dfrac{1}{2} \nabla^2 f \big(\xi^\eps_1 (y),\xi^\eps_2 (y),y\big)\big(q^{u+\eps \bar{u},\alpha}(y) - q^{u,\alpha}(y),\eps \bar{u}(y),0 \big)^2 \right\Vert ,
\end{multline}
where for any $y \in [a,b]$:
\begin{itemize}
\item $\xi^\eps_1 (y) \in [q^{u,\alpha}(y),q^{u+\eps \bar{u},\alpha}(y)] \subset [\Vert q^{u,\alpha} \Vert_\infty -C_1,\Vert q^{u,\alpha} \Vert_\infty +C_1 ]^d$ according to Lemma \ref{lemappA1}. Then, there exists $M_1 \geq 0$ such that $\xi^\eps_1 (y) \in [-M_1,M_1]^d$ for any $\abs{\eps} < 1$ and any $y \in [a,b]$;
\item $\xi^\eps_2 (y) \in [u(y),u(y)+\eps \bar{u}(y)] \subset [-M_2,M_2]^m$ with $M_2 \geq 0$ independent of $\vert \eps \vert < 1$ and $y \in [a,b]$ since $u$ and $\bar{u}$ are continuous on $[a,b]$.
\end{itemize}
Since $\nabla^2 f$ is continuous on the compact $[-M_1,M_1]^d \times [-M_2,M_2]^m \times [a,b]$, there exists $M_3 \geq 0$ such that for any $\vert \eps \vert < 1$ and any $y \in [a,b]$, we have:
\begin{multline}
\left\Vert \dfrac{1}{2} \nabla^2 f \big(\xi^\eps_1(y),\xi^\eps_2(y),y \big) \big( q^{u+\eps \bar{u},\alpha} (y) - q^{u,\alpha}(y),\eps \bar{u}(y),0)^2  \right\Vert \\ \leq M_3 \big( \Vert q^{u+\eps \bar{u},\alpha} (y) - q^{u,\alpha}(y) \Vert^2  + 2 \Vert q^{u+\eps \bar{u},\alpha} (y) - q^{u,\alpha}(y) \Vert \Vert \eps \bar{u}(y) \Vert + \Vert \eps \bar{u}(y) \Vert^2 \big). 
\end{multline}
Then, from Lemma \ref{lemappA1} and since $\bar{u}$ is continuous, there exists $M_4 \geq 0$ such that:
\begin{multline}
\forall \vert \eps \vert < 1, \; \forall y \in [a,b], \; \left\Vert \dfrac{1}{2} \nabla^2 f \big(\xi^\eps_1(y),\xi^\eps_2(y),y \big) \big( q^{u+\eps \bar{u},\alpha} (y) - q^{u,\alpha}(y),\eps \bar{u}(y),0)^2  \right\Vert \leq M_4 \eps^2 .
\end{multline}
Consequently, from Inequality \eqref{eq987} and since $f$ satisfies Condition \eqref{condf} (implying $\Vert \partial f / \partial x \Vert_{\infty} \leq M$), we have for any $\vert \eps \vert < 1$ and any $y \in [a,b]$:
\begin{multline}
\left\Vert f\big(q^{u+\eps \bar{u},\alpha}(y),u(y)+\eps \bar{u}(y),y\big)-f\big(q^{u,\alpha}(y),u(y),y\big) \right. \\ \left. - \eps \dfrac{\partial f}{\partial x}\big(q^{u,\alpha}(y),u(y),y\big) \times \bar{q}(y) - \eps \dfrac{\partial f}{\partial v}\big(q^{u,\alpha}(y),u(y),y\big) \times \bar{u}(y) \right\Vert  \\ \leq 2 M \Vert q^{u+\eps \bar{u},\alpha}(y)- q^{u,\alpha}(y)-\eps \bar{q}(y) \Vert +M_4 \eps^2.
\end{multline}
Finally, using the previous inequality and Inequality \eqref{eq654}, we obtain for any $\vert \eps \vert < 1$ and any $t \in [a,b]$: 
\begin{equation}
\Vert q^{u+\eps \bar{u},\alpha}(t) - q^{u,\alpha}(t) - \eps \bar{q}(t) \Vert \leq 2 M I^\alpha_- \big( \Vert  q^{u+\eps \bar{u},\alpha} - q^{u,\alpha} - \eps \bar{q} \Vert \big) (t) +  M_5 \eps^2,
\end{equation}
where $M_5 := M_4 (b-a)^\alpha / \Gamma (1+\alpha)$ is independent of $\vert \eps \vert < 1$. Finally, from the fractional Gronwall's Lemma \ref{lemgronf}, we conclude that:
\begin{equation}
\forall \vert \eps \vert < 1, \; \forall t \in [a,b], \; \Vert q^{u+\eps \bar{u},\alpha}(t) - q^{u,\alpha}(t) - \eps \bar{q}(t)\Vert \leq  M_5 \eps^2 {\rm E}_{\alpha,1} \big( 2M(b-a)^\alpha \big).
\end{equation}
Defining $C := M_5 {\rm E}_{\alpha,1} \big( 2M (b-a)^\alpha \big)$, the proof is completed.
\end{proof}

\subsection{Proof of Lemma \ref{lem1}}\label{appAd}
In this section, we prove Lemma \ref{lem1} and consequently, we use the notations and definitions given in Section \ref{section22}. \\

Let $u$, $\bar{u} \in \CC^0 ([a,b],\R^m )$ and $\bar{q} \in \CC^{[\alpha]} ( [a,b], \R^d )$ the unique global solution of \eqref{eqlcpq}. From Lemma \ref{lemappA2}, we have:
\begin{equation}\label{eqAd-1}
\forall \vert \eps \vert < 1, \; q^{u+\eps \bar{u},\alpha} = q^{u,\alpha} + \eps \bar{q} + h^{\eps},
\end{equation}
where $\Vert h^{\eps} \Vert_{\infty} \leq C \eps^2$. In particular, since $\bar{q}$ is continuous, there exists $M_1 \geq 0$ such that:
\begin{equation}
\forall \vert \eps \vert < 1, \; \forall y \in [a,b], \; [q^{u,\alpha}(y),q^{u+\eps \bar{u},\alpha}(y)] \subset [-M_1,M_1]^d.
\end{equation}
In the same way, since $u$ and $\bar{u}$ are continuous, there exists $M_2 \geq 0$ such that:
\begin{equation}
\forall \vert \eps \vert < 1, \; \forall y \in [a,b], \; [u(y),u(y)+\eps \bar{u}(y)] \subset [-M_2,M_2]^m.
\end{equation}
We have:
\begin{equation}\label{eqAd-2}
\forall \vert \eps \vert < 1, \; \LL^\alpha(u+\eps \bar{u})-\LL^\alpha(u) = \di \int_a^b L(q^{u+\eps \bar{u},\alpha},u+\eps \bar{u},t) - L(q^{u,\alpha},u,t) \; dt .
\end{equation}
With a Taylor's expansion of order $2$ with explicit remainder, we have for any $\vert \eps \vert < 1$ and any $y \in [a,b]$:
\begin{multline}
\Big\vert L\big(q^{u+\eps \bar{u},\alpha}(y),u(y)+\eps \bar{u}(y),y\big) - L\big(q^{u,\alpha}(y),u(y),y\big) \\ - \eps \dfrac{\partial L}{\partial x} \big(q^{u,\alpha}(y),u(y),y\big) \cdot \bar{q}(y) - \eps \dfrac{\partial L}{\partial v} \big(q^{u,\alpha}(y),u(y),y\big) \cdot \bar{u}(y) \Big\vert \\ \leq \Big\vert \dfrac{\partial L}{\partial x} \big(q^{u,\alpha}(y),u(y),y\big) \cdot h^{\eps}(y) \Big\vert  + \Big\vert \dfrac{1}{2} \nabla^2 L \big(\xi^\eps_1(y),\xi^\eps_2(y),y\big) \big(\eps \bar{q}(y)+h^\eps(y),\eps \bar{u}(y),0\big)^2 \Big\vert,
\end{multline}
where for any $\abs{\eps} <1$ and any $y \in [a,b]$, $\xi^\eps_1 (y) \in [q^{u,\alpha}(y),q^{u+\eps \bar{u},\alpha}(y)] \subset [-M_1,M_1]^d$ and $\xi^\eps_2 (y) \in [u(y),u(y)+\eps \bar{u}(y)] \subset [-M_2,M_2]^m$. Since $L$ is of class $\CC^2$, we obtain easily that there exists $M_3 \geq 0$ such that for any $\vert \eps \vert < 1$ and any $y \in [a,b]$: 
\begin{multline}
\Big\vert L\big(q^{u+\eps \bar{u},\alpha}(y),u(y)+\eps \bar{u}(y),y\big) - L\big(q^{u,\alpha}(y),u(y),y\big) \\ - \eps \dfrac{\partial L}{\partial x} \big(q^{u,\alpha}(y),u(y),y\big) \cdot \bar{q}(y) - \eps \dfrac{\partial L}{\partial v} \big(q^{u,\alpha}(y),u(y),y\big) \cdot \bar{u}(y) \Big\vert \leq M_3 \eps^2.
\end{multline}
Consequently, we have for any $0 < \vert \eps \vert < 1$:
\begin{equation}
\left\Vert \dfrac{L(q^{u+\eps \bar{u},\alpha},u+\eps \bar{u},t) - L(q^{u,\alpha},u,t)}{\eps} - \dfrac{\partial L}{\partial x} (q^{u,\alpha},u,t) \cdot \bar{q} -  \dfrac{\partial L}{\partial v} (q^{u,\alpha},u,t) \cdot \bar{u} \right\Vert_{\infty} \leq M_3 \eps.
\end{equation}
Hence:
\begin{equation}
\lim\limits_{\eps \rightarrow 0} \dfrac{\LL^\alpha(u+\eps \bar{u})-\LL^\alpha(u)}{\eps} =  \di \int_a^b \dfrac{\partial L}{\partial x} (q^{u,\alpha},u,t) \cdot \bar{q} + \dfrac{\partial L}{\partial v} (q^{u,\alpha},u,t) \cdot \bar{u} \; dt.
\end{equation}
The proof is completed.

\bibliographystyle{plain}

\end{document}